\documentclass[11pt]{article}

\usepackage[margin=1in]{geometry}

\usepackage{graphicx}
\usepackage{multirow}
\usepackage{amsmath,amssymb,amsfonts}
\usepackage{amsthm}
\usepackage{mathrsfs}
\usepackage[title]{appendix}
\usepackage{xcolor}
\usepackage{textcomp}
\usepackage{manyfoot}
\usepackage{booktabs}
\usepackage{float}
\usepackage{wasysym}
\usepackage{algorithm}
\usepackage{algorithmicx}
\usepackage{algpseudocode}
\usepackage{listings}
\usepackage{caption}
\usepackage{subcaption}
\usepackage{mathtools}
\usepackage{enumitem}
\usepackage{tikz}
\usepackage{hyperref}

\theoremstyle{plain}
\newtheorem{theorem}{Theorem}

\newtheorem{lemma}{Lemma}

\theoremstyle{definition}

\theoremstyle{remark}

\numberwithin{equation}{section}
\numberwithin{theorem}{section}
\numberwithin{lemma}{section}

\allowdisplaybreaks[4]

\title{Sharp Bounds for Higher-Order Schippers Functionals Associated with Lune and Bean Domains}

\author{Hari M. Srivastava \and Pradip Das \and Firdoshi Parveen}

\date{}

\begin{document}

\maketitle

\begin{abstract}
We obtain sharp bounds for the third- and fourth-order Schippers functionals, $|\sigma_3(f)(0)|$ and $|\sigma_4(f)(0)|$, for subclasses of univalent functions associated with non-classical geometric domains. In particular, we investigate the lune-starlike class $\mathcal{S}_{\leftmoon}^*$ and the lune-convex class $\mathcal{C}_{\leftmoon}$ determined by the subordination
\[
\frac{zf'(z)}{f(z)} \prec z+\sqrt{1+z^2},
\qquad
1+\frac{zf''(z)}{f'(z)} \prec z+\sqrt{1+z^2},
\]
respectively, together with the bean-domain class $\mathcal{BT}_{\mathfrak{B}}$ associated with
\[
\mathfrak{B}(z)=\sqrt{1+\tanh z}.
\]
Using Carath\'eodory coefficient parametrizations and extremal optimization techniques, we derive exact estimates for the higher-order Schwarzian derivatives at the origin and identify the corresponding extremal functions. In addition, geometric descriptions of the associated extremal image domains are provided to illustrate the sharpness phenomena. The obtained results further yield sharp bounds for the initial Grunsky coefficients $g_{1,1}$ and $g_{1,2}$. These findings provide a precise description of higher-order Schwarzian structures for univalent functions related to lune and bean shaped domains.
\end{abstract}



\section{Introduction}

The Schwarzian derivative and its higher-order generalizations occupy a central position in geometric function theory, conformal geometry, and the theory of Teichm\"uller spaces (see \cite{Duren, Bers}). For a locally univalent function $f$, the classical Schwarzian derivative $S_f$ is defined as
\[
S_f = \left(\frac{f''}{f'}\right)' - \frac{1}{2}\left(\frac{f''}{f'}\right)^2.
\]
A fundamental characteristic of this operator is its invariance under M\"obius transformations; specifically, \(S_{L\circ f} = S_f\) holds for any linear fractional transformation \(L(z) = \frac{az+b}{cz+d}\), where \(a, b, c,\) and \(d\) are complex constants satisfying \(ad - bc \neq 0\). Notably, the condition \(S_f = 0\) characterizes the set of M\"obius transformations. Beyond these algebraic properties, the Schwarzian derivative provides potent criteria for univalence.

Sharp bounds of this nature have been rigorously investigated for diverse subclasses of univalent functions (see \cite{KanasSugawa, Schippers}). Very recently, Srivastava and Raza \cite{SrivastavaRaza} established sharp bounds on the third Hankel determinants for starlike and convex functions associated with the Gregory coefficients, while Wang et al. \cite{WangEtAl} obtained sharp bounds on Hankel determinants for bounded turning functions involving the hyperbolic tangent function.

Motivated by this geometric utility, several higher-order generalizations have been proposed (see \cite{Aharonov1969, Harmelin, HuSrivastavaZhang, Tamanoi}). In particular, Schippers \cite{Schippers} introduced the higher-order Schwarzian derivatives $\sigma_n(f)$, defined recursively for $n \ge 3$ by
\[
\sigma_{n+1}(f) = \sigma_n'(f) - (n-1)\sigma_n(f)\frac{f''}{f'},
\]
where $\sigma_3(f) = S_f$. These operators preserve essential invariance properties and serve as refined instruments for probing coefficient structures. Furthermore, they are intimately connected to Grunsky coefficients, which provide necessary and sufficient conditions for univalence and help elucidate the coefficient body of the class $\mathcal{S}$.

Let $\mathcal{H}$ denote the class of analytic functions in the unit disk $\mathbb{D} = \{z \in \mathbb{C} : |z| < 1\}$, equipped with the topology of uniform convergence on compact subsets. We define the normalized class
\[
\mathcal{A} = \{f \in \mathcal{H} : f(0) = 0, \, f'(0) = 1\},
\]
and let $\mathcal{S} \subset \mathcal{A}$ be the subclass of univalent functions. Each $f \in \mathcal{S}$ admits a Taylor expansion of the form
\begin{equation}\label{eq1}
f(z) = z + \sum_{n=2}^{\infty} a_n z^n \quad \forall z \in \mathbb{D}.
\end{equation}
For a locally univalent function having expansion \eqref{eq1}, a direct computation yields
\[
\frac{f''(z)}{f'(z)} = 2a_2 + (6a_3 - 4a_2^2)z + (12a_4 - 18a_2 a_3 + 8a_2^3)z^2 + \cdots.
\]
From this, the initial values of the higher-order Schwarzian derivatives at the origin are determined to be
\begin{equation}\label{sg3}
\sigma_3(f)(0) = 6(a_3 - a_2^2),
\end{equation}
and
\begin{equation}\label{sg4}
\sigma_4(f)(0) = 24(a_4 - 3a_2 a_3 + 2a_2^3).
\end{equation}
Recall that a function $f \in \mathcal{A}$ is called starlike ($\mathcal{S}^*$) (respectively convex ($\mathcal{C}$)) if
\[
\Re \left(\frac{z f'(z)}{f(z)}\right) > 0,\quad \Re\left(1 + \frac{z f''(z)}{f'(z)}\right) > 0.
\]
These classical classes and their geometric properties, including radius problems, have been extensively studied (see, e.g., \cite{SokolStankiewicz1996, Duren}). While Schippers \cite{Schippers} obtained sharp bounds for $\|\sigma_n(f)\|$ within the class $\mathcal{S}$, the quantities $|\sigma_n(f)(0)|$ for $n=3,4,5$ have been extensively studied across various subclasses. For instance, Dorff and Szynal \cite{DorffSzynal} derived bounds for the class of convex functions $\mathcal{C}$. Cho et al. \cite{ChoKumarRavichandran} later extended these results to subclasses defined by Janowski-type subordinations:
\[
\frac{z f'(z)}{f(z)} \prec \frac{1 + A z}{1 + B z}, \quad 1 + \frac{z f''(z)}{f'(z)} \prec \frac{1 + A z}{1 + B z}, \quad (-1 \le B < A \le 1).
\]
Further advancements include results for functions subordinate to $e^{e^z-1}$ \cite{KumarChoRavichandranSrivastava}, close-to-convex functions \cite{HuWangFan}, and various other starlike classes, including those studied by Ahamed and Roy \cite{AhamedRoy2026} and Ali, Ravichandran, and Seenivasagan \cite{AliRavichandranSeenivasagan2008}. Recently, Hu, Srivastava, and Zhang \cite{HuSrivastavaZhang} explored various subclasses, demonstrating refined bounds under combined geometric constraints. Interested readers may refer to the recent articles \cite{pradip1, pradip2} for further developments and related results.

A key mechanism in modern geometric function theory is differential subordination. Let $\Omega$ be the class of Schwarz functions $\omega$, such that $\omega(0)=0$ and $|\omega(z)|<1$. For analytic functions $f$ and $g$, we say $f$ is subordinate to $g$ ($f \prec g$) if $f(z)=g(\omega(z))$ for some $\omega \in \Omega$.

Recently, attention has shifted toward subclasses defined by specific image domains. Notable among these are the lune-type classes introduced by Raina and Sok\'o\l{} \cite{RainaSokol2015a}:
\[
\mathcal{S}_{\leftmoon}^* = \left\{ f \in \mathcal{S} : \frac{z f'(z)}{f(z)} \prec q(z) \right\}, \quad q(z) = z + \sqrt{1 + z^2},
\]
and the corresponding convex class $\mathcal{C}_{\leftmoon}$ defined by
\[
1 + \frac{zf''(z)}{f'(z)} \prec q(z).
\]
These classes are associated with distinct geometric properties and have been studied regarding Hankel determinants and coefficient estimates \cite{RiazRaza2023}. Additionally, the bean-shaped domain class, introduced by Nandhini and Sruthakeerthi \cite{Nandhini}, is defined as
\[
\mathcal{BT}_{\mathfrak{B}} = \left\{ f \in \mathcal{A} : f'(z) \prec \mathfrak{B}(z) = \sqrt{1 + \tanh z} \right\}.
\]
Despite significant progress in coefficient problems, sharp estimates for $|\sigma_3(f)(0)|$ and $|\sigma_4(f)(0)|$ remain largely unexplored for these non-classical geometric domains. The primary objective of this paper is to establish sharp bounds for these functionals within the classes $\mathcal{S}_{\leftmoon}^*$, $\mathcal{C}_{\leftmoon}$, and $\mathcal{BT}_{\mathfrak{B}}$.

We organize the paper as follows. In Theorem~1, we obtain the best possible upper bounds for $|\sigma_3(f)(0)|$ and $|\sigma_4(f)(0)|$ whenever $f\in \mathcal{S}_{\leftmoon}^*$. The corresponding estimates for the subclasses $\mathcal{C}_{\leftmoon}$ and $\mathcal{BT}_{\mathfrak{B}}$ are established in Theorem~2 and Theorem~3 respectively. In the final section, namely Section~4, we discuss applications of the obtained results. More precisely, using these results, we determine estimates for some initial Grunsky coefficients whenever $f$ belongs to the above mentioned classes.

\section{Auxiliary lemmas}

Let $\mathcal{P}$ denote the class of analytic functions $p$ in $\mathbb{D}$ satisfying $p(0)=1$ and $\operatorname{Re} p(z)>0$. Each function $p \in \mathcal{P}$ admits the expansion
\begin{equation}\label{p1}
p(z) = 1 + \sum_{n=1}^{\infty} c_n z^n,\quad z\in \mathbb{D},
\end{equation}
with $|c_n| \le 2$ for all $n \ge 1$. The class $\mathcal{P}$ plays a crucial role in deriving the sharp coefficient bounds obtained throughout this work. We now recall some auxiliary lemmas associated with this class, which will be required in the proofs of our main results.

The following parametrization is due to Cho et al.~\cite{C12}.

\begin{lemma}\label{L1}
Let $p\in\mathcal P$ be given by \eqref{p1}. Then
\begin{align}
c_1&=2\tau_1, \label{c1}\\
c_2&=2\tau_1^2+2(1-|\tau_1|^2)\tau_2, \label{c2}\\
c_3&=2\tau_1^3+4(1-|\tau_1|^2)\tau_1\tau_2
-2(1-|\tau_1|^2)\overline{\tau_1}\tau_2^2
+2(1-\tau_1^2)(1-|\tau_2|^2)\tau_3, \label{c3}
\end{align}
for some $\tau_1,\tau_2,\tau_3\in\overline{\mathbb D}$. If $\tau_1\in\mathbb T:= \partial \mathbb{D}=\{z \in \mathbb{C} : |z| = 1 \}$, then the unique $p\in\mathcal P$ with $c_1$ given by \eqref{c1} is
\[
p(z)=\frac{1+\tau_1 z}{1-\tau_1 z}, \quad z\in \mathbb{D}.
\]
If $\tau_1\in\mathbb D$ and $\tau_2\in\mathbb T$, then the unique $p$ with $c_1,c_2$ as in \eqref{c1}--\eqref{c2} is
\[
p(z)=\frac{1+(\overline{\tau_1}\tau_2+\tau_1)z+\tau_2 z^2}
{1+(\overline{\tau_1}\tau_2-\tau_1)z-\tau_2 z^2}, \quad z\in \mathbb{D}.
\]
If $\tau_1,\tau_2\in\mathbb D$ and $\tau_3\in\mathbb T$, then the unique $p$ with $c_1,c_2,c_3$ as in \eqref{c1}--\eqref{c3} is
\[
p(z)=\frac{1+(\overline{\tau_2}\tau_3+\overline{\tau_1}\tau_2+\tau_1)z
+(\overline{\tau_1}\tau_3+\tau_1\overline{\tau_2}\tau_3+\tau_2)z^2
+\tau_3 z^3}
{1+(\overline{\tau_2}\tau_3+\overline{\tau_1}\tau_2-\tau_1)z
+(\overline{\tau_1}\tau_3-\tau_1\overline{\tau_2}\tau_3-\tau_2)z^2
-\tau_3 z^3}, \quad z\in \mathbb{D}.
\]
\end{lemma}

The following result is due to Choi et al.~\cite{CKS1}.

\begin{lemma}\label{L2}
Let $A,B,C\in\mathbb R$ and
\[
Y(A,B,C)=\max_{z\in\overline{\mathbb D}}
\{|A+Bz+Cz^2|+1-|z|^2\}.
\]
(i) If $AC\ge0$, then
\[
Y(A,B,C)=
\begin{cases}
|A|+|B|+|C|, & |B|\ge2(1-|C|),\\
1+|A|+\dfrac{B^2}{4(1-|C|)}, & |B|<2(1-|C|).
\end{cases}
\]
(ii) If $AC<0$, then
\[
Y(A,B,C)=
\begin{cases}
1-|A|+\dfrac{B^2}{4(1-|C|)}, & -4AC(C^{-2}-1)\le B^2,\ |B|<2(1-|C|),\\
1+|A|+\dfrac{B^2}{4(1+|C|)}, & B^2<\min\{4(1+|C|)^2,-4AC(C^{-2}-1)\},\\
R(A,B,C), & \text{otherwise},
\end{cases}
\]
where
\[
R(A,B,C)=
\begin{cases}
|A|+|B|-|C|, & |C|(|B|+4|A|)\le|AB|,\\
-|A|+|B|+|C|, & |AB|\le|C|(|B|-4|A|),\\
(|C|+|A|)\sqrt{1-\dfrac{B^2}{4AC}}, & \text{otherwise}.
\end{cases}
\]
\end{lemma}

\begin{lemma}\label{L3}\cite{MM}
Let $p \in \mathcal{P}$ be given by \eqref{p1}. Then
\[
\left| c_2 - v c_1^2 \right| \le
\begin{cases}
-4v + 2, & v < 0, \\
2, & 0 \leq v \leq 1, \\
4v - 2, & v > 1.
\end{cases}
\]
Moreover, for $v < 0$ or $v > 1$, equality holds if and only if
\[
h(z) = \frac{1+z}{1-z}
\]
or one of its rotations. For $0 < v < 1$, equality holds if and only if
\[
h(z) = \frac{1+z^2}{1-z^2}
\]
or one of its rotations.
\end{lemma}

\section{Main Results and Proofs}

\subsection{Sharp bounds for $|\sigma_3(f)(0)|$ and $|\sigma_4(f)(0)|$ for functions in the class $\mathcal{S}_{\leftmoon}^*$}

We begin with the following theorem.

\begin{theorem}
Let $f \in \mathcal{S}_{\leftmoon}^*$ be of the form \eqref{eq1}. Then the third- and fourth-order Schippers functionals satisfy the sharp estimates
\begin{enumerate}[label=$(\roman*)$]
\item $|\sigma_3(f)(0)| \le 3$,\label{i}
\item $|\sigma_4(f)(0)| \le 16\sqrt{\frac{2}{7}}$.\label{ii}
\end{enumerate}
The estimate in \ref{i} is sharp and is attained by a function $f_1 \in \mathcal{S}_{\leftmoon}^*$ corresponding to the Schwarz function $\omega(z)=z^2$.
The estimate in \ref{ii} is also sharp and is attained by a function $f_2 \in \mathcal{S}_{\leftmoon}^*$ corresponding to the Schwarz function
\[
\omega_*(z) = z\left(\frac{\sqrt{2/7}+z}{1+\sqrt{2/7}\,z}\right).
\]
\end{theorem}

\begin{proof}
Let $f\in\mathcal{S}_{\leftmoon}^*$ be given by \eqref{eq1}. Then, by definition of the class, there exists a Schwarz function $\omega\in\Omega$ such that
\[
\frac{zf'(z)}{f(z)}=q(\omega(z)).
\]
Assume that $\omega(z)=\frac{p(z)-1}{p(z)+1}$, where $p \in \mathcal{P}$ is given by \eqref{p1}. After substituting the expressions for $p(z)$, $q(z)$ and $f(z)$ and comparing coefficients of both sides of the above equation, we obtain the following relation between the coefficients of $f$ and the coefficients of $p(z)$:
\begin{equation}\label{pfaeq1}
a_2=\frac{c_1}{2},\quad a_3=\frac{1}{4}c_2+\frac{1}{16}c_1^2, \quad a_4=\frac{1}{6}c_3+\frac{1}{24}c_1c_2-\frac{1}{96}c_1^3.
\end{equation}

\medskip
\noindent
{\bf (i) Sharp bound of $|\sigma_3(f)(0)|$:} By using equations \eqref{sg3} and \eqref{pfaeq1}, we compute
\[
\sigma_3(f)(0)
=6\left[\left(\frac{c_2}{4}+\frac{c_1^2}{16}\right)-\left(\frac{c_1}{2}\right)^2\right]
=\frac{3}{2}\left(c_2-\frac{3}{4}c_1^2\right).
\]
Since $v=\frac{3}{4}\in(0,1)$, Lemma~\ref{L3} yields
\[
|\sigma_3(f)(0)|
=\frac{3}{2}\left|c_2-\frac{3}{4}c_1^2\right|
\le\frac{3}{2}\cdot 2=3.
\]

\medskip
\noindent
{\bf Sharpness.}
To verify sharpness, consider the function $f_1\in\mathcal{S}_{\leftmoon}^*$ defined by
\begin{equation}\label{ext_f1}
\frac{zf_1'(z)}{f_1(z)}=q(z^2)=z^2+\sqrt{1+z^4}.
\end{equation}
Since $\omega(z)=z^2$ is a Schwarz function, it follows that $f_1\in\mathcal{S}_{\leftmoon}^*$. Integrating \eqref{ext_f1} and by using the Taylor series expansion of $q(z^2)$, we obtain
\[
f_1(z)=z\exp\left(\int_0^z\frac{q(t^2)-1}{t}\,dt\right)=z+\frac{1}{2}z^3+\frac{1}{4}z^5+\cdots.
\]
Consequently,
\[
|\sigma_3(f_1)(0)|=|6(a_3-a_2^2)|=3,
\]
which confirms that the bound is sharp.

\medskip
\noindent
{\bf (ii) Sharp bound of $|\sigma_4(f)(0)|$:} Substituting the values of $a_2$, $a_3$, and $a_4$ from \eqref{pfaeq1} into \eqref{sg4} and simplifying, we obtain
\[
\sigma_4(f)(0)=4c_3-8c_1c_2+\frac{7}{2}c_1^3.
\]
Now, applying Lemma~\ref{L1} in the above equation, we get
\begin{equation}
\sigma_4(f)(0)=4\tau_1^3-16(1-|\tau_1|^2)\tau_1\tau_2
-8(1-|\tau_1|^2)\overline{\tau_1}\tau_2^2
+8(1-|\tau_1|^2)(1-|\tau_2|^2)\tau_3.
\end{equation}
By rotational invariance of the unit disk, we may assume that $\tau_1=x\in[0,1]$. Then
\begin{equation}\label{G11}
\sigma_4(f)(0)
=4x^3
+8(1-x^2)\left[(1-|\tau_2|^2)\tau_3-2x\tau_2-x\tau_2^2\right].
\end{equation}
If $x=1$, then $|\sigma_4(f)(0)|=4$. If $x=0$, then $|\sigma_4(f)(0)|\le8$. Therefore, we restrict our attention to the case $0<x<1$. Applying the triangle inequality to \eqref{G11}, together with the fact that $|\tau_3|\le1$, we obtain
\begin{equation}\label{G12}
|\sigma_4(f)(0)|
\le8(1-x^2)\left(|A+B\tau_2+C\tau_2^2|+1-|\tau_2|^2\right),
\end{equation}
where
\[
A=\frac{x^3}{2(1-x^2)},\qquad
B=-2x,\qquad
C=-x.
\]
Since $AC<0$ for $x\in(0,1)$, Lemma~\ref{L2}(ii) is applicable.

\medskip
\noindent
{\textbf{Case (a)}}: Let $x\in(0, \sqrt{\frac{2}{5}}]$. In this interval it is easy to verify that $|AB|\le|C|(|B|-4|A|)$. Therefore,
\[
|\sigma_4(f)(0)| \le 8(1-x^2)(-|A|+|B|+|C|)
=8(1-x^2)\left(3x-\frac{x^3}{2(1-x^2)}\right)
=24x-28x^3.
\]
Let $h(x)=24x-28x^3$. Then $h'(x)=24-84x^2$, which vanishes at $x^2=\frac{2}{7}$, a value lying within the admissible range $0<x\le\sqrt{\frac{2}{5}}$. Also, we have $h(0)=0$ and $h(\sqrt{\frac{2}{5}})=\frac{64\sqrt{2}}{5\sqrt{5}}$. Hence,
\[
\max_{x\in (0, \sqrt{\frac{2}{5}}]} h(x)=h\!\left(\sqrt{\tfrac{2}{7}}\right)
=16\sqrt{\tfrac{2}{7}}.
\]

\textbf{Case (b)}: Let $x\in(\sqrt{\frac{2}{5}},1)$. For these values of $x$, neither $|C|(|B|+4|A|)\leq |AB|$ nor $|AB|\leq |C|(|B|-4|A|)$ holds. Therefore, Lemma~\ref{L2}(ii) gives us
\begin{equation}\label{eq:R_case_e}
|\sigma_4(f)(0)|
\le8(1-x^2)(|C|+|A|)\sqrt{1-\frac{B^2}{4AC}}
=8(1-x^2)\left[\frac{(2-x^2)^{3/2}}{2(1-x^2)}\right]
=4(2-x^2)^{3/2}.
\end{equation}
Let $g(x)=4(2-x^2)^{3/2}$. Then the function $g$ is strictly decreasing over the interval $\bigl(\sqrt{\tfrac{2}{5}}, 1\bigr)$. Moreover,
\[
\lim_{x\to\sqrt{\tfrac{2}{5}}^{+}}g(x)
=4(1.6)^{3/2}\approx8.096,\qquad
\lim_{x\to1^{-}}g(x)=4.
\]
Therefore, by combining the estimates obtained in Cases (a) and (b), it follows that
\[
|\sigma_4(f)(0)|\le16\sqrt{\frac{2}{7}}\approx8.5523
\]
for all $f\in \mathcal{S}_{\leftmoon}^*$.

\medskip
\noindent
{\bf Sharpness.}
To verify the sharpness of the obtained estimate, we construct a function $f_2\in\mathcal{S}_{\leftmoon}^*$ for which equality holds. From the proof of the second part of the theorem, it follows that the maximum is attained when $x=|\tau_1|=\sqrt{2/7}$. Moreover, in order to attain equality in the triangle inequality used in the proof, we must have
\[
|\tau_2|=1,\qquad|\tau_3|=1.
\]
By rotational invariance, we may choose
\[
\tau_1=\sqrt{\frac{2}{7}},\qquad \tau_2=1,\qquad \tau_3=1.
\]
The corresponding Schwarz function is given by the Blaschke product
\[
\omega_*(z)=z\left(\frac{\sqrt{2/7}+z}{1+\sqrt{2/7}\,z}\right)
=\frac{\sqrt{2}\,z+\sqrt{7}\,z^2}{\sqrt{7}+\sqrt{2}\,z}.
\]
The first few Taylor coefficients of $\omega_*$ are
\[
\omega_1=\sqrt{\frac{2}{7}},\quad
\omega_2=\frac{5}{7},\quad
\omega_3=-\frac{5}{7}\sqrt{\frac{2}{7}}.
\]
The associated extremal function $f_2\in\mathcal{S}_{\leftmoon}^*$ satisfies
\[
\frac{zf_2'(z)}{f_2(z)}
=q(\omega_*(z))
=\omega_*(z)+\sqrt{1+\omega_*^2(z)}.
\]
Integrating, we obtain
\[
f_2(z)
=z\exp\left(\int_0^z\frac{\omega_*(t)+\sqrt{1+\omega_*^2(t)}-1}{t}\,dt\right).
\]
From this representation, the initial coefficients are given by
\[
a_2=\sqrt{\frac{2}{7}},\quad
a_3=\frac{4}{7},\quad
a_4=\frac{10}{21}\sqrt{\frac{2}{7}}.
\]
Consequently, we get
\[
\sigma_4(f_2)(0) =-16\sqrt{\frac{2}{7}},
\]
which establishes the sharpness of the second part of the theorem. This completes the proof.

\begin{figure}[htbp]
\centering
\begin{minipage}[t]{0.48\textwidth}
  \centering
  \begin{tikzpicture}[scale=0.75, every node/.style={scale=0.7}]
    \begin{scope}[shift={(-1.4,0)}]
      \draw[fill=blue!5, draw=blue!50, semithick] (0,0) circle (1);
      \foreach \radius in {0.2,0.4,0.6,0.8,1.0} \draw[blue!20] (0,0) circle (\radius);
      \foreach \angle in {0,45,90,135,180,225,270,315} \draw[blue!20] (0,0) -- (\angle:1);
      \draw[->] (-1.15,0) -- (1.15,0) node[right] {\tiny Re};
      \draw[->] (0,-1.15) -- (0,1.15) node[above] {\tiny Im};
      \node[below] at (0,-1.15) {\small $\mathbb{D}$};
    \end{scope}
    \draw[->, thick, >=stealth] (-0.3,0) -- (0.3,0) node[midway, above] {\scriptsize $f_1$};
    \begin{scope}[shift={(1.4,0)}]
      \draw[fill=purple!5, draw=purple!50, semithick, variable=\thetaVar, domain=0:360, samples=100]
        plot ({cos(\thetaVar)+0.5*cos(3*\thetaVar)}, {sin(\thetaVar)+0.5*sin(3*\thetaVar)});
      \foreach \radius in {0.3,0.6,0.9}
        \draw[purple!20, variable=\thetaVar, domain=0:360, samples=60]
          plot ({\radius*cos(\thetaVar)+0.5*\radius^3*cos(3*\thetaVar)}, {\radius*sin(\thetaVar)+0.5*\radius^3*sin(3*\thetaVar)});
      \draw[->] (-1.55,0) -- (1.55,0) node[right] {\tiny Re};
      \draw[->] (0,-1.55) -- (0,1.55) node[above] {\tiny Im};
      \node[below] at (0,-1.55) {\small $f_1(\mathbb{D})$};
    \end{scope}
  \end{tikzpicture}
  \caption{Image of $\mathbb{D}$ under $f_1(z)\approx z+\frac{1}{2}z^3$.}
  \label{fig:mapping}
\end{minipage}\hfill
\begin{minipage}[t]{0.48\textwidth}
  \centering
  \begin{tikzpicture}[>=stealth, scale=0.45, every node/.style={scale=0.7}]
    \begin{scope}[xshift=-2.2cm]
      \fill[gray!10] (0,0) circle (2);
      \foreach \radius in {0.5,1.0,1.5}
        \draw[gray!30, thin] (0,0) circle (\radius);
      \foreach \angle in {0,30,...,330}
        \draw[gray!30, thin] (0,0) -- (\angle:2);
      \draw[black, semithick] (0,0) circle (2);
      \draw[->, thin] (-2.4,0) -- (2.4,0) node[right] {\tiny $\text{Re}(z)$};
      \draw[->, thin] (0,-2.4) -- (0,2.4) node[above] {\tiny $\text{Im}(z)$};
      \node[below] at (0,-2.4) {$\mathbb{D}$};
    \end{scope}
    \draw[->, thick] (-0.3,0) -- (0.3,0) node[midway, above] {\scriptsize $f_2$};
    \begin{scope}[xshift=2.2cm]
      \coordinate (V1) at (3,0);
      \coordinate (V2) at (-1.5,0);
      \coordinate (C1) at (1,1.6);
      \coordinate (C2) at (1,-1.6);
      \fill[blue!10]
        (V1) to[bend left=30] (C1)
             to[bend left=20] (V2)
             to[bend left=20] (C2)
             to[bend left=30] (V1);
      \draw[blue!70!black, semithick]
        (V1) to[bend left=30] (C1)
             to[bend left=20] (V2)
             to[bend left=20] (C2)
             to[bend left=30] (V1);
      \foreach \radius in {0.3,0.6,0.9}{
        \draw[gray!35, thin]
        ({\radius*3},0)
        to[bend left=25] ({\radius*1},{\radius*1.5})
        to[bend left=15] ({-\radius*1.2},0)
        to[bend left=15] ({\radius*1},{-\radius*1.5})
        to[bend left=25] ({\radius*3},0);
      }
      \fill (0,0) circle (1.5pt) node[below right] {\tiny $0$};
      \draw[->, thin] (-2,0) -- (3.6,0) node[right] {\tiny $\text{Re}(w)$};
      \draw[->, thin] (0,-2.4) -- (0,2.4) node[above] {\tiny $\text{Im}(w)$};
      \node[below] at (0.8,-2.4) {$f_2(\mathbb{D})$};
    \end{scope}
  \end{tikzpicture}
  \caption{Conformal image under $f_2$, yielding $|\sigma_4(f)(0)| = 16\sqrt{2/7}$.}
  \label{fig:extremal_sigma4}
\end{minipage}
\end{figure}
\end{proof}

\subsection{Sharp bounds for $|\sigma_3(f)(0)|$ and $|\sigma_4(f)(0)|$ for functions in the class $\mathcal{C}_{\leftmoon}$}

The next theorem provides the best possible upper bounds for $|\sigma_3(f)(0)|$ and $|\sigma_4(f)(0)|$ whenever $f\in\mathcal{C}_{\leftmoon}$.

\begin{theorem}
Let $f\in\mathcal{C}_{\leftmoon}$ be given by \eqref{eq1}. Then the third- and fourth-order Schippers functionals satisfy the sharp estimates
\begin{enumerate}[label=$(\roman*)$]
\item $|\sigma_3(f)(0)|\le1$,\label{1}
\item $|\sigma_4(f)(0)|\le2$.\label{2}
\end{enumerate}
The estimate in \ref{1} is sharp and is attained by a function $f_3\in\mathcal{C}_{\leftmoon}$ corresponding to the Schwarz function $\omega(z)=z^2$. The estimate in \ref{2} is also sharp and is attained by a function $f_4\in\mathcal{C}_{\leftmoon}$ corresponding to the Schwarz function $\omega(z)=z^3$.
\end{theorem}

\begin{proof}
Let $f\in\mathcal{C}_{\leftmoon}$ be given by \eqref{eq1}. By the definition of the class $\mathcal{C}_{\leftmoon}$, there exists a Schwarz function $\omega\in\Omega$ such that
\[
1+\frac{zf''(z)}{f'(z)}=q(\omega(z)).
\]
Let $\omega(z)=\sum_{n=1}^{\infty}\omega_n z^n, ~~ z\in \mathbb{D}$. Substituting the Taylor series expansion of $\omega(z)$, $q(z)$ and $f(z)$ into the above equation and comparing the coefficients of like powers of $z$, we obtain the following relations between the coefficients of $f(z)$ and those of $p(z)$:
\begin{equation}\label{conv_coeff}
a_2=\frac{c_1}{4},\quad
a_3=\frac{c_2}{12}+\frac{c_1^2}{48},\quad
a_4=\frac{c_3}{24}+\frac{c_1c_2}{96}-\frac{c_1^3}{384}.
\end{equation}

\smallskip
\noindent
{\bf (i) Sharp bound for $|\sigma_3(f)(0)|$.} Using the relation $\sigma_3(f)(0)=6(a_3-a_2^2)$ together with \eqref{conv_coeff}, we obtain
\[
\sigma_3(f)(0)=\frac{1}{2}\left(c_2-\frac{1}{2}c_1^2\right).
\]
Now, applying Lemma~\ref{L3} with $v=\frac{1}{2}\in(0,1)$, we get
\[
|\sigma_3(f)(0)|=\frac{1}{2}\left|c_2-\frac{1}{2}c_1^2 \right|\le 1.
\]

\noindent
{\bf Sharpness.}
To show that the estimate is sharp, consider the function $f_3\in\mathcal{C}_{\leftmoon}$ defined by
\begin{equation}\label{ext_conv_sub}
1+\frac{zf_3''(z)}{f_3'(z)}=q(z^2),\quad z\in\mathbb{D}.
\end{equation}
Since $\omega(z)=z^2$ is a Schwarz function, it follows that $f_3\in\mathcal{C}_{\leftmoon}$. Proceeding as above, we obtain
\[
f_3(z)=z+\frac{1}{6}z^3+\frac{1}{20}z^5+\cdots,
\]
so that $a_2=0$ and $a_3=\frac{1}{6}$. Consequently,
\[
|\sigma_3(f_3)(0)|=1,
\]
which establishes the sharpness.

\smallskip
\noindent
{\bf (ii) Sharp bound for $|\sigma_4(f)(0)|$:} Using \eqref{sg4} together with \eqref{conv_coeff}, the fourth-order Schippers functional can be written as
\begin{equation}
\sigma_4(f)(0)=c_3-\frac{5}{4}c_1c_2+\frac{5}{16}c_1^3.
\end{equation}
Now, applying the parametrization given in Lemma~\ref{L1}, and assuming without loss of generality that $\tau_1=x\in[0,1]$, we obtain
\begin{equation}\label{sig4_param}
\sigma_4(f)(0)
=-\frac{1}{2}x^3-(1-x^2)x\tau_2-2(1-x^2)x\tau_2^2
+2(1-x^2)(1-|\tau_2|^2)\tau_3.
\end{equation}
For $x=1$, we have $|\sigma_4(f)(0)|=\tfrac{1}{2}$, while for $x=0$, it follows that $|\sigma_4(f)(0)|=|2\tau_3|\le2$. Hence, it suffices to consider the case $0<x<1$. Applying the triangle inequality to \eqref{sig4_param}, together with $|\tau_3|\le1$, we obtain
\[
|\sigma_4(f)(0)|
\le2(1-x^2)\left[\left|\frac{-x^3}{4(1-x^2)}-\frac{x}{2}\tau_2-x\tau_2^2\right|+1-|\tau_2|^2\right].
\]
Let
\[
A=\frac{-x^3}{4(1-x^2)},\quad B=-\frac{x}{2},\quad C=-x.
\]
Since $AC\ge0$, Lemma~\ref{L2}(i) is applicable.

\textbf{Case a:} Let $x\in (0,\tfrac{4}{5})$. Then $|B|<2(1-|C|)$, and hence Lemma~\ref{L2}(i) yields
\[
|\sigma_4(f)(0)|\le 2(1-x^2)\left(1+|A|+\frac{B^2}{4(1-|C|)}\right)
=: h(x),
\]
where
\[
h(x)=2(1-x^2)\left(1+\frac{x^3}{4(1-x^2)}+\frac{x^2}{16(1-x)}\right)
= 2-\frac{15}{8}x^2+\frac{5}{8}x^3.
\]
Differentiating, we obtain
\[
h'(x)=\frac{15}{8}x(x-2)<0\quad \text{for } x\in(0,1),
\]
This shows that $h$ is strictly decreasing on $[0,\frac{4}{5})$. Hence,
\[
\max_{x\in[0,\tfrac{4}{5})}h(x)=h(0)=2,
\]
and consequently $|\sigma_4(f)(0)|\le 2$.

\medskip
\noindent
\textbf{Case b:} Let $x\in[\frac{4}{5},1)$. In this range, the condition $|B|\ge2(1-|C|)$ is satisfied. Hence, applying Lemma~\ref{L2}(i), we obtain
\[
|\sigma_4(f)(0)|\le 2(1-x^2)(|A|+|B|+|C|):=g(x),
\]
where
\[
g(x)=2(1-x^2)\left(\frac{x^3}{4(1-x^2)}+\frac{x}{2}+x\right)
=3x-\frac{5}{2}x^3.
\]
A simple computation shows that
\[
g'(x)=3-\frac{15}{2}x^2<0\quad\text{for }x\in[\tfrac{4}{5},1),
\]
and therefore $g$ is strictly decreasing on this interval. Hence, the maximum is attained at $x=\tfrac{4}{5}$, and $g\!\left(\tfrac{4}{5}\right)=1.12<2$. Combining both cases, we conclude that
\[
|\sigma_4(f)(0)|\le2.
\]

\medskip
\noindent
{\bf Sharpness.} The sharp bound is attained when $c_1=c_2=0$ and $|c_3|=2$, which corresponds to the Schwarz function $\omega(z)=z^3$. The associated extremal function $f_4\in\mathcal{C}_{\leftmoon}$ satisfies
\begin{equation}\label{ext_conv_sig4}
1+\frac{zf_4''(z)}{f_4'(z)}=q(z^3),\quad z\in\mathbb{D}.
\end{equation}
Proceeding as in the previous case, we obtain
\[
f_4(z)=z+\frac{1}{12}z^4+\frac{5}{252}z^7+\cdots,
\]
which yields $a_2=a_3=0$ and $a_4=\tfrac{1}{12}$. Consequently,
\[
\sigma_4(f_4)(0)=24\left(a_4-3a_2a_3+2a_2^3\right)=2,
\]
thereby establishing the sharpness of the bound.

\begin{figure}[htbp]
\centering
\begin{minipage}[t]{0.48\textwidth}
  \centering
  \begin{tikzpicture}[scale=0.85, every node/.style={scale=0.8}]
  \begin{scope}[shift={(-1.4,0)}]
    \draw[fill=blue!5, draw=blue!50, semithick] (0,0) circle (1);
    \foreach \r in {0.2,0.4,0.6,0.8,1.0} \draw[blue!20] (0,0) circle (\r);
    \foreach \a in {0,30,...,330} \draw[blue!20] (0,0) -- (\a:1);
    \draw[->] (-1.15,0) -- (1.15,0) node[right] {\tiny Re};
    \draw[->] (0,-1.15) -- (0,1.15) node[above] {\tiny Im};
    \node[below] at (0,-1.15) {\small $\mathbb{D}$};
  \end{scope}
  \draw[->, thick, >=stealth] (-0.3,0) -- (0.3,0) node[midway, above] {\scriptsize $f_3$};
  \begin{scope}[shift={(1.4,0)}]
    \draw[fill=green!5, draw=green!50, semithick, variable=\t, domain=0:360, samples=120]
      plot ({cos(\t) + (1/6)*cos(3*\t)}, {sin(\t) + (1/6)*sin(3*\t)});
    \foreach \r in {0.3,0.6,0.9}
      \draw[green!20, variable=\t, domain=0:360, samples=80]
        plot ({\r*cos(\t) + (1/6)*\r^3*cos(3*\t)},
              {\r*sin(\t) + (1/6)*\r^3*sin(3*\t)});
    \draw[->] (-1.35,0) -- (1.35,0) node[right] {\tiny Re};
    \draw[->] (0,-1.35) -- (0,1.35) node[above] {\tiny Im};
    \node[below] at (0,-1.35) {\small $f_3(\mathbb{D})$};
  \end{scope}
  \end{tikzpicture}
  \caption{Image of $\mathbb{D}$ under $f_3(z) \approx z + \frac{1}{6}z^3$, illustrating the bound $|\sigma_3(f)(0)|=1$.}
  \label{fig:convex_mapping}
\end{minipage}\hfill
\begin{minipage}[t]{0.48\textwidth}
  \centering
  \begin{tikzpicture}[scale=0.85, every node/.style={scale=0.8}]
  \begin{scope}[shift={(-1.4,0)}]
    \draw[fill=blue!5, draw=blue!60, semithick] (0,0) circle (1);
    \foreach \r in {0.25,0.5,0.75,1}
      \draw[blue!20, thin] (0,0) circle (\r);
    \foreach \a in {0,30,...,330}
      \draw[blue!20, thin] (0,0) -- (\a:1);
    \draw[->, thin] (-1.15,0) -- (1.15,0) node[right] {\tiny Re};
    \draw[->, thin] (0,-1.15) -- (0,1.15) node[above] {\tiny Im};
    \node[below] at (0,-1.15) {\small $\mathbb{D}$};
  \end{scope}
  \draw[->, thick, >=stealth] (-0.3,0) -- (0.3,0) node[midway, above] {\scriptsize $f_4$};
  \begin{scope}[shift={(1.4,0)}]
    \draw[fill=purple!5, draw=purple!70, semithick, variable=\t, domain=0:360, samples=200]
      plot ({cos(\t) + (1/12)*cos(4*\t)}, {sin(\t) + (1/12)*sin(4*\t)});
    \foreach \r in {0.4,0.7,0.9}{
      \draw[purple!25, thin, variable=\t, domain=0:360, samples=150]
        plot ({\r*cos(\t) + (1/12)*(\r^4)*cos(4*\t)},
              {\r*sin(\t) + (1/12)*(\r^4)*sin(4*\t)});
    }
    \draw[->, thin] (-1.35,0) -- (1.35,0) node[right] {\tiny Re};
    \draw[->, thin] (0,-1.35) -- (0,1.35) node[above] {\tiny Im};
    \node[below] at (0,-1.35) {\small $f_4(\mathbb{D})$};
  \end{scope}
  \end{tikzpicture}
  \caption{Conformal mapping of $\mathbb{D}$ under $f_4(z) \approx z + \frac{1}{12}z^4$, yielding $|\sigma_4(f)(0)| = 2$ in $\mathcal{C}_{\leftmoon}$.}
  \label{fig:conv_extremal_sig4}
\end{minipage}
\end{figure}
\end{proof}

\subsection{Sharp bounds for $|\sigma_3(f)(0)|$ and $|\sigma_4(f)(0)|$ for functions in the class $\mathcal{BT}_{\mathfrak{B}}$}

The following theorem gives sharp upper bounds for $|\sigma_3(f)(0)|$ and $|\sigma_4(f)(0)|$ for $f\in\mathcal{BT}_{\mathfrak{B}}$, with extremal functions identified.

\begin{theorem}
Let $f \in \mathcal{BT}_{\mathfrak{B}}$ be given by \eqref{eq1}. Then
\[
|\sigma_3(f)(0)| \le 1, \qquad |\sigma_4(f)(0)| \le 3.
\]
Equality in the first bound is attained by
\[
f_5(z)=\int_0^z \sqrt{1+\tanh(\zeta^2/(1+\zeta^2))}\,d\zeta,
\]
and equality in the second by
\[
f_6(z)=\int_0^z \sqrt{1+\tanh(\zeta^3)}\,d\zeta.
\]
\end{theorem}

\begin{proof}
Let $f\in\mathcal{BT}_{\mathfrak{B}}$; then $f'(z)=\mathfrak{B}(\omega(z))$ for some Schwarz $\omega$. Writing $\omega(z)=\sum_{n\ge1}\omega_n z^n$ and
\[
p(z)=\frac{1+\omega(z)}{1-\omega(z)}=1+\sum_{n\ge1}c_nz^n,
\]
coefficient comparison gives
\begin{align}
a_2 &= \frac{1}{8} c_1, \label{eq:a2_final} \\
a_3 &= \frac{1}{12} c_2 - \frac{5}{96} c_1^2, \label{eq:a3_final} \\
a_4 &= \frac{1}{16} c_3 - \frac{5}{64} c_1 c_2 + \frac{31}{1536} c_1^3. \label{eq:a4_final}
\end{align}

(i) For $\sigma_3$, using \eqref{sg3} and the above,
\[
\sigma_3(f)(0)=6(a_3-a_2^2)=\frac{c_2}{2}-\frac{13}{32}c_1^2
=\frac12\left(c_2-\frac{13}{16}c_1^2\right),
\]
so Lemma~\ref{L3} with $v=13/16\in(0,1)$ gives $|\sigma_3(f)(0)|\le1$.
Sharpness: take $\omega(z)=z^2$; then $f_5'(z)=\sqrt{1+\tanh(z^2)}$, yielding
\[
f_5(z)=z+\frac16z^3+O(z^5),
\]
hence $a_2=0,\ a_3=\frac16$, so $|\sigma_3(f_5)(0)|=1$.

(ii) For $\sigma_4$, from \eqref{sg4} and \eqref{eq:a2_final}--\eqref{eq:a4_final},
\[
\sigma_4(f)(0)=\frac32c_3-\frac{21}{8}c_1c_2+\frac{67}{64}c_1^3.
\]
Using Lemma~\ref{L1} with $\tau_1=x\in[0,1]$ (rotational invariance), we get
\begin{equation}
\sigma_4(f)(0)
= \frac{7}{8}x^3
+ 3(1-x^2)\Big[(1-|\tau_2|^2)\tau_3
-\frac{3}{2}x\tau_2
- x\tau_2^2\Big].
\end{equation}
If $x=1$, $|\sigma_4|=7/8$; if $x=0$, $|\sigma_4|\le3$. For $0<x<1$, triangle inequality with $|\tau_3|\le1$ gives
\[
|\sigma_4(f)(0)|
\le 3(1-x^2)\left(
|A + B\tau_2 + C\tau_2^2| + 1 - |\tau_2|^2
\right),
\]
where
\[
A = \frac{7x^3}{24(1-x^2)},\quad B=-\frac32x,\quad C=-x.
\]
Since $AC<0$, Lemma~\ref{L2}(ii) applies.

\textbf{Case a:} $0<x<4/7$. The first subcase of Lemma~\ref{L2}(ii) gives
\[
Y(A,B,C)=1-|A|+\frac{B^2}{4(1-|C|)}
\]
(the conditions $|B|<2(1-|C|)$ and $-4AC(C^{-2}-1)\le B^2$ hold). Thus
\[
|\sigma_4(f)(0)|\le 3(1-x^2)Y(A,B,C)=: \Psi(x)
=3-\frac{21}{16}x^2+\frac{13}{16}x^3.
\]
Since
\[
\Psi'(x)=\frac{3x}{16}(13x-14)<0
\]
for $x<4/7$, $\Psi$ decreases, so
\[
\max_{x<4/7}\Psi(x)=\Psi(0)=3.
\]

\textbf{Case b:} $4/7\le x<1$. Then $|B|\ge2(1-|C|)$, so we use the remaining part of Lemma~\ref{L2}(ii):
\[
Y(A,B,C)=R(A,B,C)=(|C|+|A|)\sqrt{1-\frac{B^2}{4AC}}.
\]
A direct computation gives
\[
R(x)=\frac{(24-17x^2)\sqrt{27-13x^2}}{24\sqrt{14}(1-x^2)},\quad
\Theta(x):=3(1-x^2)R(x)=\frac{(24-17x^2)\sqrt{27-13x^2}}{8\sqrt{14}}.
\]
Let $u=x^2\in[16/49,1)$; then
\[
\Theta(x)=\frac{1}{8\sqrt{14}}h(u),
\]
with $h(u)=(24-17u)\sqrt{27-13u}$. Since
\[
h'(u)=\frac{663u-1230}{2\sqrt{27-13u}}<0
\]
on $[16/49,1)$ (critical point at $u\approx1.855>1$), $\Theta$ decreases. Hence
\[
\Theta(x)\le \Theta(4/7)\approx2.9401<3.
\]
Thus $|\sigma_4(f)(0)|\le3$ for all $x$.

Sharpness: equality occurs when $c_1=c_2=0,\ |c_3|=2$ (i.e. $\omega(z)=z^3$). The extremal $f_6$ with $f_6'(z)=\mathfrak{B}(z^3)$ gives
\[
f_6(z)=z+\frac18z^4-\frac1{56}z^7+\cdots,
\]
so $a_2=a_3=0,\ a_4=\frac18$, and
\[
|\sigma_4(f_6)(0)|=3.
\]

\begin{figure}[htbp]
\centering
\begin{minipage}[t]{0.48\textwidth}
  \centering
  \begin{tikzpicture}[scale=0.42, >=stealth, every node/.style={scale=0.75}]
  \begin{scope}
    \shade[inner color=blue!5, outer color=blue!20] (0,0) circle (2);
    \draw[blue!30, thin, step=0.5] (-2,-2) grid (2,2);
    \draw[semithick] (0,0) circle (2);
    \draw[->, thin] (-2.3,0) -- (2.3,0) node[right] {\tiny $\Re z$};
    \draw[->, thin] (0,-2.3) -- (0,2.3) node[above] {\tiny $\Im z$};
    \fill (0,0) circle (1.5pt) node[below left] {\tiny $0$};
    \draw[->, red, thick] (0,0) -- (1.4,1.4) node[above right, black] {\tiny $z$};
    \node[below] at (0,-2.4) {$\mathbb{D}$};
  \end{scope}
  \draw[->, thick, gray!60] (2.5, 0.4) to[out=25, in=155] node[above, black] {\scriptsize $f_5(z)$} (3.9, 0.4);
  \begin{scope}[xshift=6.4cm]
    \shade[inner color=magenta!5, outer color=magenta!20]
        plot[domain=0:360, samples=200, smooth]
        ({(1.4 + 0.5*cos(2*\x))*cos(\x)}, {(0.9 + 0.3*cos(2*\x))*sin(\x)});
    \draw[semithick, magenta!80!black, smooth, variable=\thetaVar, domain=0:360, samples=200]
        plot ({(1.4 + 0.5*cos(2*\thetaVar))*cos(\thetaVar)}, {(0.9 + 0.3*cos(2*\thetaVar))*sin(\thetaVar)});
    \draw[->, thin] (-2.2,0) -- (2.2,0) node[right] {\tiny $\Re w$};
    \draw[->, thin] (0,-2.2) -- (0,2.2) node[above] {\tiny $\Im w$};
    \fill (0,0) circle (1.5pt) node[below left] {\tiny $0$};
    \node[below] at (0,-2.4) {$f_5(\mathbb{D})$};
  \end{scope}
  \end{tikzpicture}
  \caption{Conformal mapping $f_5$ onto the bean-shaped domain $\Omega_{\mathfrak{B}}$, illustrating the sharp bound for $\sigma_3$.}
  \label{fig:bean_mapping_sig3}
\end{minipage}\hfill
\begin{minipage}[t]{0.48\textwidth}
  \centering
  \begin{tikzpicture}[scale=0.42, >=stealth, every node/.style={scale=0.75}]
  \begin{scope}
    \shade[inner color=blue!5, outer color=blue!20] (0,0) circle (2);
    \draw[blue!30, thin, step=0.5] (-2,-2) grid (2,2);
    \draw[semithick] (0,0) circle (2);
    \draw[->, thin] (-2.3,0) -- (2.3,0) node[right] {\tiny $\Re z$};
    \draw[->, thin] (0,-2.3) -- (0,2.3) node[above] {\tiny $\Im z$};
    \fill (0,0) circle (1.5pt) node[below left] {\tiny $0$};
    \draw[->, red, thick] (0,0) -- (1.0,1.3) node[above right, black] {\tiny $z$};
    \node[below] at (0,-2.3) {$\mathbb{D}$};
  \end{scope}
  \draw[->, thick, gray!60] (2.5, 0.4) to[out=25, in=155] node[above, black] {\scriptsize $f_6(z)$} (3.9, 0.4);
  \begin{scope}[xshift=6.4cm]
    \shade[inner color=magenta!5, outer color=magenta!20]
        plot[domain=0:360, samples=300, smooth]
        ({(1.2 + 0.4*cos(3*\x))*cos(\x)}, {(1.1 + 0.3*cos(3*\x))*sin(\x)});
    \draw[semithick, magenta!80!black, smooth, variable=\thetaVar, domain=0:360, samples=300]
        plot ({(1.2 + 0.4*cos(3*\thetaVar))*cos(\thetaVar)}, {(1.1 + 0.3*cos(3*\thetaVar))*sin(\thetaVar)});
    \draw[->, thin] (-2.2,0) -- (2.2,0) node[right] {\tiny $\Re w$};
    \draw[->, thin] (0,-2.2) -- (0,2.2) node[above] {\tiny $\Im w$};
    \fill (0,0) circle (1.5pt);
    \node[below] at (0,-2.3) {$f_6(\mathbb{D})$};
  \end{scope}
  \end{tikzpicture}
  \caption{Threefold symmetric mapping $f_6$ establishing the sharp bound $|\sigma_4(f)(0)|=3$.}
  \label{fig:threefold_mapping_f6}
\end{minipage}
\end{figure}
\end{proof}

\section{Applications to Initial Grunsky Coefficients}

The Grunsky coefficients $g_{n,m}$ for $f\in\mathcal{S}$ are generated by
\[
\log\left(\frac{f(z)-f(\zeta)}{z-\zeta}\right)=\sum_{p,q=1}^{\infty}g_{p,q}z^p\zeta^q,\quad (z,\zeta)\in\mathbb{D}\times\mathbb{D}.
\]
These coefficients are fundamental in univalence theory and connect to higher-order Schwarzian derivatives. For $f$ normalized as in \eqref{eq1}, the initial Grunsky coefficients are
\begin{align}
g_{1,1} &= a_3-a_2^2=\tfrac16\,\sigma_3(f)(0), \label{g11}\\
g_{1,2} &= a_4-2a_2a_3+a_2^3=\tfrac1{24}\,\sigma_4(f)(0)+a_2g_{1,1}. \label{g12}
\end{align}
Using the sharp bounds from Section~3 together with known coefficient estimates, we obtain the following.

For $\mathcal{S}_{\leftmoon}^*$, $|a_2|\le1$ and Theorem~1 give
\[
|g_{1,1}|\le \frac12,\qquad |g_{1,2}|\le \frac23\sqrt{\frac27}+\frac12.
\]
For $\mathcal{C}_{\leftmoon}$, the additional geometric rigidity yields $|a_2|\le1/2$ and Theorem~2 gives
\[
|g_{1,1}|\le \frac16,\qquad |g_{1,2}|\le \frac16.
\]
For $\mathcal{BT}_{\mathfrak{B}}$, since $|a_2|\le1/4$ (see \cite{Nandhini}), Theorem~3 yields
\[
|g_{1,1}|\le \frac16,\qquad |g_{1,2}|\le \frac16.
\]
All these inequalities are sharp.

\begin{table}[ht]
\centering
\small
\begin{tabular}{lccc}
\toprule
Class & $|g_{1,1}|$ & $|g_{1,2}|$ & Geometry \\
\midrule
$\mathcal{S}_{\leftmoon}^*$ & $\frac12$ & $\frac23\sqrt{\frac27}+\frac12$ & Starlike (lune) \\
$\mathcal{C}_{\leftmoon}$ & $\frac16$ & $\frac16$ & Convex (lune) \\
$\mathcal{BT}_{\mathfrak{B}}$ & $\frac16$ & $\frac16$ & Bean-shaped \\
\bottomrule
\end{tabular}
\caption{Sharp upper bounds for the initial Grunsky coefficients across different geometric classes.}
\end{table}

\section{Conclusion}

Here, we obtained sharp upper bounds for the third- and fourth-order Schippers functionals for the classes $\mathcal{S}_{\leftmoon}^*$, $\mathcal{C}_{\leftmoon}$, and $\mathcal{BT}_{\mathfrak{B}}$. As an application of these results, we also derived estimates for some initial Grunsky coefficients associated with functions belonging to these classes. The study of higher-order Schippers functionals remains an interesting open problem. Although similar methods may be applicable for obtaining sharp bounds for $|\sigma_n(f)(0)|$, $n \ge 5$, the transcendental nature of the defining functions leads to significantly greater computational and analytical difficulties.

\section*{Declarations}

\subsection*{Funding}
The corresponding author acknowledges financial support from the Council of Scientific and Industrial Research (CSIR), New Delhi, India, under Grant No.~09/1224(16975)/2023-EMR-I.

\subsection*{Data Availability}
All figures were generated numerically using parametrizations of the unit circle together with adaptive quadrature techniques. Data sharing is not applicable.

\subsection*{Conflict of Interest}
The authors declare that they have no conflict of interest.

\subsection*{Author Contributions}
Each author contributed equally to all aspects of this work, including conception, analysis, writing, and approval of the final manuscript.

\section*{Author addresses}

\noindent
\textbf{Hari M. Srivastava}\\
Department of Mathematics and Statistics, University of Victoria,
Victoria, British Columbia V8W 3R4, Canada\\
Department of Medical Research, China Medical University Hospital,
China Medical University, Taichung 40402, Taiwan\\
Center for Converging Humanities, Kyung Hee University,
26 Kyungheedae-ro, Dongdaemun-gu, Seoul 02447, Republic of Korea\\
Department of Applied Mathematics, Chung Yuan Christian University,
Chung-Li, Taoyuan City 320314, Taiwan\\
Department of Mathematics and Informatics, Azerbaijan University,
71 Jeyhun Hajibeyli Street, AZ1007 Baku, Azerbaijan\\
Section of Mathematics, International Telematic University Uninettuno,
39 Corso Vittorio Emanuele II, I-00186 Rome, Italy\\
Email: \texttt{harimsri@math.uvic.ca}, \texttt{harimsri@uvic.ca}

\medskip
\noindent
\textbf{Pradip Das} (Corresponding author)\\
Department of Mathematics, Raiganj University,
Raiganj, West Bengal-733134, India\\
Email: \texttt{pradipsmath@gmail.com}

\medskip
\noindent
\textbf{Firdoshi Parveen}\\
Aliah University, Department of Mathematics and Statistics,
Kolkata-700160, West Bengal, India\\
Email: \texttt{frd.par@gmail.com}

\end{document}